\documentclass[12pt,reqno]{amsart}
\usepackage{amssymb,amsfonts,latexsym,amstext,amsmath,epsfig}
\usepackage[margin=1in]{geometry}
\usepackage{amsmath}

\newtheorem{theorem}{Theorem}

\newtheorem{corollary}[theorem]{Corollary}

\begin{document}

\title[Generalized Cops and Robbers games]{Characterizations and algorithms for generalized Cops and Robbers games}

\author{Anthony Bonato}
\address{Department of Mathematics, Ryerson University, Toronto, ON, Canada}
\email{abonato@ryerson.ca}
\author{Gary MacGillivray}
\address{Department of Mathematics and Statistics, University of Victoria, Victoria, BC, Canada}
\email{gmacgill@uvic.ca}

\thanks{The authors gratefully acknowledge support from NSERC}

\begin{abstract}
We propose a definition of generalized Cops and Robbers games where there are two players, the Pursuer and the Evader, who each move via prescribed rules. If the Pursuer can ensure that the game
enters into a fixed set of final positions, then the Pursuer wins; otherwise, the Evader wins. A relational characterization of the games where the Pursuer wins is provided. A precise formula is
given for the length of the game, along with an algorithm for computing if the Pursuer has a winning strategy whose complexity is a function of the parameters of the game. For games where the
position of one player does not affect the available moves of he other, a vertex elimination ordering characterization, analogous to a cop-win ordering, is given for when the Pursuer has a winning
strategy.
\end{abstract}

\maketitle

\section{Introduction}~\label{introduction}

Cops and Robbers, introduced over thirty years ago in \cite{AF,Q1,Q2}, is one example of a graph searching game that has received wide attention. The players consist of a set of \emph{cops} and a
\emph{robber}, who move alternatingly from vertex to vertex over a sequence of discrete time-steps. The cops attempt to capture the robber by occupying its vertex, while the robber is trying to avoid
this situation from ever occurring. The consideration of Cops and Robbers gives rise to a number of deep results and problems in structural, algorithmic, and random graph theory; for further reading,
see the book~\cite{bonatorjn} and the surveys~\cite{bonato1,bonato2,bonato3}.

Many variations of Cops and Robbers are possible, for example, where different kinds of moves are allowed or forbidden, the cops or robber have more or less power (for example, the cops could win if
they are sufficiently close to the robber). The game may also be played analogously on directed graphs or other relational structures. Our purpose here is not to survey all of these myriad variants,
but rather, to offer a single, common framework for the analysis of \emph{all} reasonably defined generalizations of Cops and Robbers which are perfect information games. We regard such games as
being played between a Pursuer and an Evader, where the Pursuer's objective is ``capturing'' the Evader, and the Evader's objective is to avoid capture forever (hence, the Pursuer plays the role of
the set of cops, and the Evader plays the role of the robber).

Our main contribution is to show that all suitably defined ``pursuit games'' admit essentially the same analysis, and
that the main tools have existed for quite some time.
The paper is organized as follows. In Section~\ref{DefnSec}, we present a general definition of generalized Cops and Robbers games and its associated bipartite state digraph, written
$\mathcal{D}_\mathcal{G}$. A labelling of the vertices of $\mathcal{D}_\mathcal{G}$ makes it possible to determine which player has a winning strategy, and also determines the length of the game.
Section~\ref{relchar} provides a relational characterization of the games where the Pursuer has a winning strategy; see Theorem~\ref{main1}.  This result generalizes the relational characterization
of cop-win graphs from~\cite{NW}, and the one given for $k$-cop-win graphs in~\cite{cm,HM}.  Our focus in the next section then shifts to those generalized Cops and Robbers games where the moves
available to either player are not restricted by the position of the other player.  For some of these games we are able to give a refinement of the relational characterization, as well as a
characterization in terms of a vertex ordering of an associated product graph.  We give an algorithm that determines if the Pursuer has a winning strategy.  In the case of Cops and Robbers with $k$
cops, this algorithm matches the time complexity of the known algorithms recognizing a winning strategy for the cops; see~\cite{bcp,bonatorjn,cm}. We emphasize that the characterizations and
algorithms we provide apply to contexts far more general than the original Cops and Robbers game; for example, they apply to games played on any kind of relational structures (such as directed
graphs, ordered sets, or hypergraphs) and with a variety of possible rules.

\section{Generalized Cops and Robbers games: definitions and examples}\label{DefnSec}

Intuition suggests that in a game of generalized Cops and Robbers, each player should have their own position at any time-step in the game.  The set of positions of the game should therefore, consist
of ordered pairs $(p_P, q_E)$, where $p_P$ is the position of player the Pursuer, and $q_E$ is the position of player the Evader.  It may be that not every possible pair of positions is allowed or
achievable from a given position. A move in a generalized Cops and Robbers game should involve a player possibly changing their own position, but not altering the position of their opponent.  The
\emph{state} of the game consists of the current position of the game, and an identifier of which player is next to move; each state is an ordered pair $((p_P, q_E), X)$, where $X \in \{P, E\}$
indicates which player is next to move. Each player, on their turn, alters the state of the game by possibly changing their position, and changing the player who  is next to move. We formalize these
notions in the following definition, which has some similarities to the definition of a combinatorial game.

A \emph{generalized Cops and Robbers game} is any discrete-time process $\mathcal{G}$ that satisfies the following rules:
\begin{enumerate}
\item There are two players named the \emph{Pursuer} and the \emph{Evader}.

\item There is perfect information.

\item There is a set $\mathcal{P}_P$  of \emph{allowed positions for the Pursuer} and a set $\mathcal{P}_E$ of \emph{allowed positions for the Evader}.  The set of \emph{positions of the game} is
    the subset $\mathcal{P} \subseteq \mathcal{P}_P \times \mathcal{P}_E$ of positions that can be achieved when moves are made according to the rules of the game.  Similarly, the set  of
    \emph{states} of the game is the subset of $\mathcal{S} \subseteq \mathcal{P} \times \{P, E\}$ such that $((p_P, q_E), X) \in \mathcal{S}$ if the position $(p_P, q_E)$, with $X$ next to move,
    can be achieved when moves are made according to the rules of the game.

\item For each state of the game and each player, there is a non-empty set of allowed moves. Each allowed move leaves the position of the other player unchanged.  We use $\mathcal{A}_P(p_P, q_E)$
    to denote the set of \emph{allowed moves for the Pursuer} when the state of the game is  $((p_P, q_E), P)$, and $\mathcal{A}_E(p_P, q_E)$ to denote the set of \emph{allowed moves for the
    Evader} when the state of the game is  $((p_P, q_E), E)$.

\item There is a set $\mathcal{I} \subseteq \mathcal{P}_P \times \mathcal{P}_E$ of \emph{allowed start positions}.  We define the set
    $\mathcal{I}_P = \{p_P: (p_P, q_E) \in \mathcal{I}\ \mathrm{for\ some}\ q_E \in \mathcal{P}_E\}$ and, for $p_P \in \mathcal{P}_P$, we define the set $\mathcal{I}_E(p_P) = \{q_E \in
    \mathcal{P}_E: (p_P, q_E) \in \mathcal{I}\}.$ The game $\mathcal{G}$ begins by the Pursuer choosing a position $p_P \in \mathcal{I}_P$, and then the Evader choosing a position $q_E \in
    \mathcal{I}_E(p_P)$.  

\item After each player has chosen its initial position, the sides move alternately with the Pursuer moving first.  Each player, on its turn, must choose an allowed move given the current state of
    the game.

\item The rules of the game specify when the Pursuer has caught the Evader. That is, there is a subset $\mathcal{F}$ of \emph{final positions}.  The Pursuer wins $\mathcal{G}$ if, at any time-step,
    the current position of the game belongs to $\mathcal{F}$.  The Evader wins if their current position never belongs to $\mathcal{F}$.
\end{enumerate}

We only consider generalized Cops and Robbers games where the set $\mathcal{P}$ of positions is finite.  All the games we consider are played over a sequence of discrete time-steps or
\emph{rounds} which are indexed by natural numbers (including $0$).

Analogously to combinatorial games (for example, see \cite{Smith}),
Generalized Cops and Robbers games may be analyzed using bipartite digraphs.
These state digraphs have appeared many
times in the Cops and Robbers and graph theory literature; see for example, \cite{acp,cm,ffn,HM}.
In the sequel we describe the method.
Later, we will link it to the relational characterization of games which are won by the Pursuer.
For a given generalized Cops and
Robbers game $\mathcal{G}$, construct the \emph{state digraph}, $\mathcal{D}_\mathcal{G}$ as follows. The vertex set $V$ is the disjoint union $\mathcal{S}_P \cup \mathcal{S}_E$, where
$\mathcal{S}_P$ is the set of states of the game when it is the Pursuer's turn to move, and $\mathcal{S}_E$ is the set of states of the game when it is the Evader's turn to move.  The directed edges
of $\mathcal{D}_\mathcal{G}$ have one end in $\mathcal{S}_P$ and the other in $\mathcal{S}_E$; there is a directed edge from vertex $s$ to vertex $t$ whenever the rules of the game allow a move so
that the state of the game changes from $s$ to $t$. After the game has started, we can imagine the state of the game being represented by a unique token placed on a vertex of
$\mathcal{D}_\mathcal{G}$. Each player, on their turn, makes a move \emph{from} the current vertex $s$ \emph{to} another vertex $t$ by sliding the token along a directed edge from $s$ to $t$. The
Pursuer wins if and only if the token is eventually on a vertex $((p_P, q_E), X)$ with $(p_P, q_E) \in \mathcal{F}$.  Otherwise, the Evader wins.

As with combinatorial games (see \cite{Smith,Steinhaus}), a labelling of the state digraph can be used
to determine if the Pursuer has a winning strategy (also see \cite{BI,HM}), and the length of the game.
We describe the labelling procedure in the case of generalized Cops and Robbers games.
Initially, the vertices $((p_P,
q_E), X)$ with $(p_P, q_E) \in \mathcal{F}$ have $\mathrm{label}(S) = 0$, and all other vertices $X$ have  $\mathrm{label}(X) = \infty$.  We use the notation $N^+(u)$ for the set of out-neighbours of
a vertex $u$ in a directed graph. The following process is then repeated until no labels change:

\begin{enumerate}
\item If $X \in \mathcal{S}_P$, then set $\mathrm{label}(X) = 1 + \min_{Y \in N^+(X)}\mathrm{label}(Y)$.
\item If $X \in \mathcal{S}_E$, then set $\mathrm{label}(X) = 1 + \max_{Y \in N^+(X)}\mathrm{label}(Y)$.
\end{enumerate}
The intuition behind the labels is that the Pursuer's optimal strategy is to move so that the game is over as quickly as possible, and the Evader's optimal strategy is to move so the game lasts as
long as possible. It is straightforward to prove by induction that the final labels are the number of moves that the game will last from a given position, assuming both sides play optimally (see
Theorem~\ref{RelationIndex}).  The Pursuer has a winning strategy has a winning strategy if and only if he can choose a starting position $p_P \in \mathcal{I}_P$ so that $\mathrm{label}(P) < \infty$
for all $q_E \in \mathcal{I}_E(p_P)$.

We next provide some examples of generalized Cops and Robbers games. The list is by no means exhaustive.

\begin{enumerate}
\item \emph{Cops and Robbers} \cite{AF,NW,Q1,Q2}. Players moving to neighbouring vertices, and the final states represent any move of a cop onto the vertex containing the robber.

\item \emph{Distance-$k$ Cops and Robbers} \cite{bcp,ccnv}. In this variation, the cops and robber move as in the original game, but the final states include the case when the robber is within
    distance $k$ of some cop, where $k$ is a fixed positive integer.

\item \emph{Tandem-win Cops and Robbers} \cite{CN1,C3}. In this variation, there are pairs of cops who must remain within distance one of each other. This alters only the set of allowed positions of the cops.

\item \emph{Cops and Robbers with traps} \cite{cn}. Some devices are deployed by the cops so that if the robber enters a node with a ``trap'' he loses the game. Here, we restrict the allowed moves
    of the robber. In particular, if $T$ is the set of vertices with traps, then $\mathcal{F}$ contains $V_P \times T$ (along with the usual final states in Cops and Robbers).

\item \emph{Eternal domination} \cite{ghh}. Fix a graph $G$. The positions for the Pursuer are the vertices of $G$.  The positions for the Evader are the $k$-subsets of $V(G)$, where $k\ge 1$ is an
    integer. The vertices in these $k$-subsets are regarded as each holding a guard. The positions for the Pursuer are regarded as vertices where the Evader must locate a guard on their next move.
    The allowed starting positions are the pairs $(x, X)$ with $x \in X.$  On their turn, the Pursuer can move to any vertex from any other vertex. On their turn, the Evader can move from position
    $X_1$ to position $Y_1$ when every guard in $X_1$ can slide along an edge (maybe a loop) so that the resulting configuration is $Y_1$ (or whatever the variant of the game dictates; for example,
    only one guard may move).  The final states are the pairs $(x, Y)$ with $x \not \in Y$.  The graph $G$ has an \emph{eternal dominating set of cardinality} $k$ if and only if the Evader has a
    winning strategy.

\item \emph{Revolutionaries and spies} \cite{butter,mp}. For fixed positive integers $r$ and $s$, there is a set of $r$ \emph{revolutionaries} (who are the Evaders) and a team of $s$ \emph{spies}
    (who are the Pursuers). For a fixed positive integer $m$, a \emph{meeting} is a set of at least $m$ revolutionaries occupying a vertex; a meeting is \emph{unguarded} if there is no spy at that
    vertex. The revolutionaries begin the game by occupying some set of vertices, and then the spies do the same. Players may move in each subsequent round to adjacent vertices or remain at their
    current vertex. The revolutionaries win if at the end of some round there is an unguarded meeting. The spies win if they can prevent an unguarded meeting from ever occurring. The final
    positions are those where the revolutionaries form an unguarded meeting.

\item \emph{Seepage} \cite{bmp,CFFMN}. This game is played on a directed acyclic graph, with a single source (vertex of in-degree zero) and a fixed number of sinks (vertices of out-degree zero).
    the Evader is called \emph{sludge} and there is some number of \emph{greens} (the Pursuer). The sludge begins at the source and, on each turn, moves along directed edges to any out-neighbours
    not protected by the greens, or simply does not move. On their turn, the greens protect some vertices not occupied by sludge by moving to them, if possible, (one vertex per green; not
    necessarily an adjacent vertex); once a vertex is protected, it remains that way to the end of the game.  The final positions are those where the greens have prevented sludge from entering a
    sink.  A related problem is $S$-Fire \cite{FM}.

\end{enumerate}

\section{A relational characterization}~\label{relchar}

We now show that a relational characterization similar to, but slightly more general than, the one given in \cite{NW, cm} for Cops and Robbers
holds for all generalized Cops and Robbers games.
In games like Cops and Robbers, where the position of one player does not affect the collection of moves available to the
opposing player, the
relational characterization has on additional properties; see Section~\ref{pind}.

Let $\mathcal{G}$ be a generalized Cops and Robbers game.  Similarly to \cite{cm,NW}, define a non-decreasing sequence of relations from $\mathcal{P}_E$ to $\mathcal{P}_P$, indexed by natural
numbers, as follows:
\begin{enumerate}
\item $q_E \preceq_0 p_P$ if and only if $(p_P, q_E) \in \mathcal{F}$.
\item Suppose $\preceq_0, \preceq_1, \ldots \preceq_{i-1}$ have all been defined for some $i \geq 1$.  Define $q_E \preceq_i p_P$ if $(p_P, q_E) \in \mathcal{F}$, or if $((p_P, q_E), E) \in
    \mathcal{S}$,  and for every $x_E \in \mathcal{A}_E(p_P, q_E)$ either $(p_P, x_E) \in \mathcal{F}$ or there exists $w_P \in \mathcal{A}_P(p_P, x_E)$ such that  $x_E \preceq_j  w_P$ for some $j
    < i$.
\end{enumerate}
By definition $\preceq_{i}$ contains $\preceq_{i-1}$ for each $i \geq 1$.  Since $\mathcal{P}_E \times \mathcal{P}_P$ is finite, there exists $t$ such that $\preceq_t\ =\ \preceq_k$ for all $k \geq
t$.  Define $\preceq\ =\ \preceq_t$.

We next use this sequence of relations to derive a method of determining the winner of a given generalized Cops and Robbers game.  For $i > 0$, the intuition behind the definition of $\preceq_i$ can
be phrased as: \emph{if the game is not over, for every move that the Evader can make from this position, either the game ends or the Pursuer has a response that leads to a win after he has moved $j$
more times}. Since the Pursuer is next to move once a start position $(p_P, q_E) \in \mathcal{I}$ has been determined, it is not necessary that $q_E \preceq p_P$ in order for the Pursuer to have a
winning strategy. Instead he needs to be able to choose their initial position $p_P$ in such a way that no matter which initial position is chosen by the Evader, he has a move to a position $w_P \in
\mathcal{A}_P(p_P, q_E)$ such that $q_E \preceq w_P$.

\smallskip
\begin{theorem}~\label{main1}
The Pursuer has a winning strategy in the generalized Cops and Robbers game $\mathcal{G}$ if and only if there exists $p_P \in \mathcal{I}_P$ such that, for all $q_E \in \mathcal{I}_E(p_P)$, either
$(p_P, q_E) \in \mathcal{F}$ or there exists $w_P \in \mathcal{A}_P(p_P, q_E)$ such that $q_E \preceq w_P$.
\end{theorem}
\begin{proof}
We prove the contrapositive of the forward implication. Suppose that the condition does not hold.  Then for all $p_P \in \mathcal{I}_P$, there exists $q_E \in \mathcal{I}_E(p_P)$ such that there is
no $w_P \in \mathcal{A}_P(p_P, q_E)$ with $q_E \preceq w_P$. If the Evader chooses such a position then, no matter the position $v_P$ to which the Pursuer moves, the state of the game is $((v_P,
q_E), E)$ with $q_E \not\preceq v_P$. By definition of the sequence of relations, the Evader can move to $x_E \in \mathcal{A}_E(v_P, q_E)$ for which there is no $y_P \in \mathcal{A}_P(v_P,x_E)$ such
that $x_E \preceq y_P$.  By induction, it never occurs that the Pursuer is on $p_P$ and the Evader is on $q_E$ such that $q_E \preceq _0 p_P$; hence, the Evader has a winning strategy.

For the reverse implication, we first prove by induction that if it is the Evader's move and the state of the game is $(p_P, q_E)$ with $q_E \preceq_i p_P$, then the Pursuer wins after making $i$
more moves.  This is clear when $i = 0$.  Assume the statement holds for all natural numbers $j$ with $0 \leq j \leq i-1$.  By definition of the sequence of relations, since $q_P \preceq_i p_E$ and
$q_P \not\preceq_{i-1} p_E$, the Evader has a move to a position $w_E \in \mathcal{A}_E(p_E, q_P)$ such that, for all positions $x_P \in \mathcal{A}(p_P, w_E)$ to which the Pursuer can subsequently
move, $i-1 = \min \{j: w_E \preceq_j x_P\}$. Hence, by definition of $\preceq_i$, the statement holds for $i.$

Suppose there exists $p_P \in \mathcal{I}_P$ such that, for all $q_E \in \mathcal{I}_E(p_P)$, either $(p_P, q_E) \in \mathcal{F}$ or there exists $w_P \in \mathcal{A}_P(p_P, q_E)$ such that $q_E
\preceq w_P$.  After the start state has been determined it is the Pursuer's move.  The condition guarantees that if the game is not over, then it is always possible for the Pursuer to move to a
state from which he has a winning strategy.   Hence, the reverse implication holds.
\end{proof}

The definition of $\preceq$ is similar to the labelling procedure described in Section~\ref{DefnSec}, except that the indices in the sequence take into account a move by each player. Assuming that
the Pursuer has a winning strategy, the sequence of relations can be used to determine the length of the game in terms of the number of moves that he must make in order to win assuming both sides
play optimally.   For an allowed start position $(p_P, q_E) \in \mathcal{I}$, define
$$\ell(p_P, q_E) =
\begin{cases}
\min\{j: q_E \preceq_j p_P\}& \text{ if } q_E \preceq_j p_P,\\
1 + \max_{w_P \in \mathcal{A}_P(p_P, q_E))} \min\{j: q_E \preceq_j w_P\}&  \text{ otherwise.}\\
\end{cases}
$$

We have the following corollary.
\smallskip
\begin{corollary}~\label{corlength}
Suppose the Pursuer has a winning strategy in the generalized Cops and Robbers game $\mathcal{G}$. Assuming optimal play the length of the game is
$$\min_{p_P \in \mathcal{I}_P}\, \max_{q_E \in \mathcal{I}_E(p_P)}  \ell(p_P, q_E).$$
\end{corollary}
\begin{proof}
 As noted in the proof of Theorem \ref{main1}, by definition of the increasing sequence of relations, if $x \preceq_i y$ and $x \not\preceq_{i-1} y$, then the Evader has a move to a position $z \in
\mathcal{A}_E(y, x)$ such that $$i-1 = \min_{w \in \mathcal{A}_P(z, x)} \{j: z \preceq_j w\}.$$  Suppose the Pursuer chooses $p_P \in \mathcal{I}_P$, and the Evader subsequently chooses an allowed
starting position $q_E \in \mathcal{I}_E(p_P)$.  If $q_E \preceq_i p_P$, then assuming optimal play, the process ends in $i$ more moves by the Pursuer.  If, on the other hand, $q_E \not\preceq_i
p_P$, then assuming optimal play, the Pursuer will move to a position $w_P$ for which the value of $j$ such that $q_E \preceq_j w_P$ is minimized, and the process ends in $j+1$ more moves by the
Pursuer.
\end{proof}

We now establish a direct connection between the labelling of the state digraph and the sequence of relations, thus, linking the two methods of analysis.

\smallskip
\begin{theorem}
Let $\mathcal{G}$ be a generalized Cops and Robbers game. The label of the vertex $((p_P, q_E), E) \in V(\mathcal{D}_\mathcal{G})$ equals $k$ if and only if $q_E \preceq_{\lceil \frac{k}{2} \rceil}
p_P$ and $q_E \not\preceq_{\lceil \frac{k}{2} \rceil - 1} p_P$  when $k \geq 1$. \label{RelationIndex}
\end{theorem}
\begin{proof}
 The result will follow if we can establish the following items: (i) if $q_E \preceq p_E$, then $((p_P, q_E), E)$ is labelled with a natural number; and (ii) if the label of the vertex $((p_P, q_E),
E) \in V(\mathcal{D}_\mathcal{G})$ equals $k$, then  $q_E \preceq_{\lceil \frac{k}{2} \rceil} p_P$ and $q_E \not\preceq_{\lceil \frac{k}{2} \rceil - 1} p_P$  when $k \geq 1$.%

To prove (i) we argue by induction on $n$ that if $q_E \preceq_n p_P$, then $((p_P, q_E), E)$ is labelled with a natural number. The statement is true when $n=1$ by the labelling procedure.  Suppose,
for some $m \geq 1$, that it is true when $n=m-1$, and that  $q_E \preceq_m p_P$.

Then for every $w_E \in \mathcal{A}_E(p_P, q_E)$ either $(p_P, w_E) \in \mathcal{F}$, or there exists $x_P \in \mathcal{A}_P(p_P, w_E)$ such that $w_E \preceq_j x_P$ for some $j < m$.  Hence, by the
labelling procedure, for every $w_E \in \mathcal{A}_E(p_P, q_E)$ either $\mathrm{label}((p_P, w_E), P) = 0$, or there exists $x_P \in \mathcal{A}_P(p_P, w_E)$ such that $(x_P, w_E)$ is labelled with
a natural number. In the latter case, by the labelling procedure again, the vertex $((p_P, w_E), P)$ is labelled with a natural number,  since every out-neighbour of the vertex $(p_, q_E), E)$ is
labelled with a natural number.  The statement now follows by induction.

The proof of (ii)  is by induction on $k$.   It follows from the definitions that the statement is true when $k = 0$.  If $k = 1$, then we have $(p_P, q_E) \not\in \mathcal{F}$, but
for every $w_E \in \mathcal{A}_E(p_P, q_E)$, the position $(p_P, w_E) \in \mathcal{F}$.
Hence, by definition of the sequence of relations,
$q_E \preceq_1 p_P$.  Since $(p_P, q_E) \not\in \mathcal{F}$, we have $q_E \not\preceq_{0} p_P$.  Therefore, the statement is also true when $k = 1$.

Suppose, for some $t \geq 2$, the statement is true when $ 0 \leq k \leq t-1$, and we have that $\mathrm{label}((p_P, q_E), E) = t$. Then by the definition of the labelling,
$$t-1 \geq \max_{w_E \in \mathcal{A}_E(p_P, q_E)}  \mathrm{label}((p_P, w_E) , P)$$
and equality is achieved at least once.
By definition of the labelling again, for any the Evader position $w_E \in \mathcal{A}_E(p_P, q_E)$,
$$t-2 \geq \min_{x_P \in \mathcal{A}_P(p_P, w_E)}  \mathrm{label}((w_P, w_E) , E)$$
and equality is achieved at least once for each $w_E$ with $\mathrm{label}((p_P, w_E) , P) = t-1$. Thus, by the induction hypothesis, for every $w_E \in \mathcal{A}_E(p_P, q_E)$, either $(p_P, w_E)
\in \mathcal{F}$, or there exists $x_P \in \mathcal{A}_P(p_P, w_E)$ such that $w_E \preceq_j x_P$ for some $j < t$, and at least one such $j$ equals $t-1$. Therefore, $q_P \preceq_{t} p_P$ and $q_E
\not\preceq_{t - 1} p_P$.  The statement now follows by induction.
\end{proof}
We note that the labels on the vertices in $\mathcal{S}_P$ are uniquely determined by the labelling procedure once the labels on the vertices in $\mathcal{S}_E$ are known.

\section{Position independent games}\label{pind}

In games like Cop and Robbers and many of its variants, each player has the same set of available moves irrespective of the position of the other player (some of these moves may be much better than
others, however; for example, a bad but allowed move of the robber is to move to a neighbour of a cop).

We define a generalized Cops and Robbers game $\mathcal{G}$ to be \emph{position independent} if for all $p_P \in \mathcal{P}_P$ and $q_E \in \mathcal{P}_E$, whenever $(p_P, q_E) \not\in \mathcal{F}$
the set $\mathcal{A}_P(p_P, q_E)$ depends only on $p_P$ and the set $\mathcal{A}_E(p_P, q_E)$ depends only on $q_E$.  That is, if the game is not over, from any position of the game, the set of
available moves for a player does not depend on the position of the other player. A game is called \emph{position dependent} if it is not position independent.   For example, Cops and Robbers is
position independent, while Seepage is position dependent.

In a position independent game, the set of positions of the game can be taken to be $\mathcal{P}_P \times \mathcal{P}_E$.  Furthermore, because of position independence one can define sets
$\mathcal{M}_P$ and $\mathcal{M}_E$ of \emph{allowed moves} for the Pursuer and the Evader, respectively;  for example $(p_P, p^ \prime_P) \in \mathcal{M}_P$ if and only if $p^\prime_P \in A_P(p_P,
q_E)$ (for all $q_E \in \mathcal{P}_E$).  The ordered pairs $G_P = (P_P, M_P)$ and $G_E = (P_E, M_E)$ are the \emph{position digraphs} for the Pursuer and the Evader, respectively. For all $p_P \in
\mathcal{P}_P$ and $q_E \in \mathcal{P}_E$, we have that $\mathcal{A}_P(p_P, q_E) = N^+_{G_P}(p_P)$ and $\mathcal{A}_E(p_P, q_E) = N^+_{G_E}(r_E)$. Position independent generalized Cops and Robbers
games are thus, games played on digraphs (in a certain sense).

For digraphs $G$ and $H$, the \emph{categorical product}, written $G \times H$, has vertices $V(G)\times V(H)$ and a directed edge $(a,b)(c,d)$ whenever $(a,c)\in E(G)$ and $(b,d) \in E(H)$. The
categorical product $\mathcal{R}_{\mathcal{G}} = G_P \times G_E$ is the \emph{round summary digraph} for the game $\mathcal{G}$.  The vertices of this digraph represent the state of the game at the
end of each round consisting of a move by each player, and assuming the Evader is next to move.  The edges represent possible transitions between states in consecutive rounds.  Not all of the edges
necessarily make sense.  For example, it may be that one cannot make the transition from state $s_1$ to state $s_2$ without a state in $\mathcal{F}$ arising in the middle of the round.  Directed
edges between final positions also make no sense. Technically, then, the position digraph should be a subgraph of $\mathcal{R}_{\mathcal{G}}$.  We will not make this distinction because the extra
edges will not affect our treatment.

The following consequence of Theorem \ref{main1} describes situations where the relation being trivial characterizes the games won by the Pursuer.

\smallskip
\begin{corollary}~\label{cor1}
Let $\mathcal{G}$ be a position independent generalized Cops and Robbers game. If $G_P$ is strongly connected and there exists $X \subseteq P_P$ such that $\mathcal{I} = X \times P_E$, then the
Pursuer has a winning strategy in $\mathcal{G}$ if and only if $\preceq\ =\  V(\mathcal{R}_{\mathcal{G}}) =P_P\times P_E$.
\end{corollary}
\begin{proof}
 The proof of the reverse direction is analogous to the proof of the forward direction in Theorem~\ref{main1}, and so is omitted. For the forward direction, suppose the Pursuer has a winning strategy.
Then there exists an allowed starting position, say $p_P$, for the Pursuer such that for all $q_E \in P_E$, either $(p_P, q_E) \in \mathcal{F}$ or there exists $w_P \in N^+_{G_P}(p_P)$ and a natural
number $j$ such that $q_E \preceq_j w_P$. Therefore, there is a natural number $i$ such that $q_E \preceq_i p_P$ for all $q_E \in P_E$.  It now follows by induction that all $q_E \preceq_{i+k} x$ for
all $q_E \in P_E$ for all $x$ joined to $p_P$ by a directed path of length $k$.  The result now follows as $G_P$ is strongly connected.
\end{proof}

The position independent case gives rise to an algorithm which is a polynomial time function of the number of positions, assuming the number of allowed moves of the Pursuer and the Evader is not too
large.

\smallskip
\begin{theorem}~\label{complexity}
Let $\mathcal{G}$ be a position independent generalized Cops and Robbers game. Given the graphs $G_P$ and $G_E$, if $N^+_{G_P}(p_P)$ and $N^+_{G_E}(p_E)$ can be obtained in time
$O(f_P(\mathcal{P}_P))$ and $O(f_E(\mathcal{P}_E))$, respectively, then there is a $$O(|\mathcal{P}_P| \cdot |\mathcal{P}_E|  \cdot f_P(|\mathcal{P}_P|)  \cdot f_E(|\mathcal{P}_E|))$$ algorithm to
determine if the Pursuer has a winning strategy, and the length of the game assuming optimal play. \label{PolyTime}
\end{theorem}

\begin{proof}
We describe an algorithm to determine if the Pursuer has a winning strategy in a position independent game.  Define the matrix $\mathcal{M}_{\mathcal{G}}$, with rows indexed by elements of
$\mathcal{P}_P$ and columns indexed by elements of $\mathcal{P}_E$, to record the sequence of relations.  If $(p_P, q_E) \in \mathcal{F}$, then the  $(p_P, q_E)$ entry of $\mathcal{M}_{\mathcal{G}}$
is zero.  For $i > j$, the $(p_P, q_E)$ entry of $\mathcal{M}_{\mathcal{G}}$ equals $i$ if for every $x_E \in N^+_{G_E}(q_E)$, there exists $y_P \in N^+_{G_P}(p_P)$ such that the $(x_E, y_P)$ entry
of $\mathcal{M}_{\mathcal{G}}$ equals $j$. The matrix $\mathcal{M}_{\mathcal{G}}$ has $|P_P \times P_E|$ entries.  For all positive integers $i$, and until no entries in the matrix change, each entry
in $\mathcal{M}_{\mathcal{G}}$ not yet assigned a value must be tested to see if it can be set to $i$.  It may happen that for each $x_E \in N^+_{G_E}(q_E)$, there exists $y_P \in N^+_{G_P}(p_P)$
such that the $(x_E, y_P)$ entry of $\mathcal{M}_{\mathcal{G}}$ equals $j < i$.  For each $i$ we can update $\mathcal{M}_{\mathcal{G}}$ row-by-row.

By Corollary \ref{cor1}, the Pursuer has a winning strategy if and only if every entry of $\mathcal{M}_{\mathcal{G}}$ is eventually assigned a value.  By definition of the sequence of relations, one
can conclude that the Pursuer has a winning strategy as soon as some row of $\mathcal{M}_{\mathcal{G}}$ has a value for each entry.  More information can be obtained by filling in the entire matrix.
By Corollary \ref{corlength}, the length of the game, assuming optimal play, is $\max_{p_P \in \mathcal{P}_P} \,\min_{q_E \in \mathcal{P}_E}\quad \mathcal{M}_{\mathcal{G}}(p_P, q_E).$

We now consider the complexity of this algorithm. It takes $|\mathcal{F}| \leq |\mathcal{P}_P \times \mathcal{P}_E|$ steps to set the entries corresponding to the final positions to zero.   For each
$i$, every entry not yet assigned a value must be tested to see if it can be set to $i$.  The procedure attempts to fill in $\mathcal{M}_{\mathcal{G}}$ row-by-row.  For each row $p_P$, suppose it
takes $f_P(|\mathcal{P_P}|)$ time to obtain $N^+_{G_P}(p_P)$.  Possibly all $|\mathcal{P}_E|$ entries in the row must be tested.  To test whether the entry in column $q_E$ can be assigned a value
requires obtaining the list of all possible the Evader moves from $q_E$.  Suppose this takes time $f_E(|\mathcal{P}_E|)$.  Then processing row $p_P$ takes  time at most $O(f_P(|\mathcal{P}_P|) \cdot
|\mathcal{P}_E|  \cdot f_E(|\mathcal{P}_E|)). $  There are $|\mathcal{P}_P|$ rows and possibly as many as $|\mathcal{P}_P \times \mathcal{P}_E|$ values for $i$.  Hence, given the graphs $G_P$ and
$G_E$, the algorithm takes time at most $O(|\mathcal{P}_P| \cdot  |\mathcal{P}_E|  \cdot f_P(|\mathcal{P}_P|) \cdot f_E(|\mathcal{P}_E|)),$ as desired.
\end{proof}

As an application of Theorem~\ref{complexity}, consider the game of Cops and Robber on a reflexive graph $G$ with a fixed positive number $k$ of cops.  The graph $G_E$ is the reflexive graph $G$ and
the graph $G_P$ is the $k$-fold categorical product of $G$ with itself.  We can take $f_P$ and $f_E$ to be the maximum degree of $G_P$ and $G_E$, respectively.  Hence, $f_P$ is $O(n^k)$ and $f_E$ is
$O(n)$. Finally, $|\mathcal{F}| = 2n^2$ and $|\mathcal{P}_E \times \mathcal{P}_P| = n^{k+1}$. Thus, given $G_P$ and $G_E$, we can decide if the Pursuer has a winning strategy, and the length of the
game assuming optimal play, in time
$$O(n^{k+1}(n \cdot n^k))=O(n^{2k+2}).$$
Note, however, a succinct description of the game would consist only of the graph $G$ on $n$ vertices. The graph $G_P$ would need to be constructed, which can be accomplished in time $O(n^{2k})$. The
overall complexity does not change because
 $$O(n^{2k} +n^{k+1}(n \cdot n^k))=O(n^{2k+2}).$$
The bound  $O(n^{2k+2})$
matches the bound on the complexity of the algorithms presented in \cite{bcp,HM}. Note that if $k$ is not fixed, then determining if $k$ cops
have a winning strategy in Cops and Robbers is \textbf{EXPTIME}-complete \cite{k}.

\medskip

We finish by providing a vertex elimination ordering for generalized Cops and Robbers games, analogous to cop-win orderings characterizing cop-win graphs and the generalization for $k$-cop-win graphs
and variations of the Cops and Robbers game  \cite{cm,NW}. In a position independent generalized Cops and Robbers game $\mathcal{G}$, a vertex $(p_P,q_E)$ of $\mathcal{R}_{\mathcal{G}}$ is
\emph{removable with respect to} $X \subseteq V(\mathcal{R}_{\mathcal{G}})$ if either
\begin{enumerate}
\item $(p_P,q_E) \in \mathcal{F}$, or
\item for every $x_E \in N^+_{G_E}(q_E)$ either $(p_P, x_E) \in \mathcal{F}$ or there exists $y_P \in N^+_{G_P}(p_P)$ such that $(y_P, x_E)$ is in $X$.
\end{enumerate}
A \emph{removable vertex ordering} for $\mathcal{R}_G$ is a sequence of vertices in which each element $(p_P, q_E)$ is removable with respect to the set $X$ of vertices that precede it in the
sequence, and there exists $p_P \in \mathcal{I}_P$ so that, for all $q_E \in \mathcal{I}_E(p_P)$, either $(p_P, q_E) \in \mathcal{F}$ or there exists $w_P \in \mathcal{A}(p_P, q_E)$ such that $(w_P,
q_E)$ belongs to the sequence.

If the Pursuer has a winning strategy, then listing the vertices of $\mathcal{S}_G$ so that the pairs in $\preceq_0$ are followed by those in $\preceq_1 \setminus \preceq_0$, then those in $\preceq_2
\setminus \preceq_1$, and so on, is a removable vertex ordering for $\mathcal{S}_G$.  Other orders may be possible; for example, it may be that not all  pairs in $\preceq_0$ need to be listed before
a pair in  $\preceq_1$ can be listed.

We now have the following vertex-ordering characterization of games where the Pursuer has a winning strategy.

\smallskip
\begin{corollary}~\label{ordering}
The Pursuer has a winning strategy in the position independent generalized Cops and Robbers game $\mathcal{G}$ if and only if $\mathcal{R}_{\mathcal{G}}$ admits a removable vertex ordering.
\end{corollary}
\begin{proof}
 Suppose that the Pursuer has a winning strategy.  Then by Theorem~\ref{main1}, there exists $p_P \in \mathcal{I}_P$ such that for all $q_E \in \mathcal{I}_E(p_P)$ either $(p_P, q_E) \in \mathcal{F}$,
or there exists $w_P \in \mathcal{A}(p_P, q_E)$ with $q_E \preceq w_P$.  Listing the vertices of $\mathcal{S}_G$ so that the pairs in $\preceq_0$ are followed by those in $\preceq_1 \setminus
\preceq_0$, then those in $\preceq_2 \setminus \preceq_1$, and so on, is a removable vertex ordering for $\mathcal{R}_G$.

On the other hand, suppose $\mathcal{R}_{\mathcal{G}}$ admits a removable vertex ordering.  Then there is a sequence vertices such that each element $(p_P, q_E)$ is removable with respect to the set
$X$ of vertices that precede it in the sequence, and there exists $p_P \in \mathcal{I}_P$ so that  for all $q_E \in \mathcal{I}_E(p_P)$ either $(p_P, q_E) \in \mathcal{F}$ or there exists $w_P \in
N^+_{G_P}(p_P)$ such that $(w_P, q_E)$ belongs to the sequence.  the Pursuer chooses $p_P$ for their starting position.  No matter which position $q_E$ the Evader chooses, if the game is not over,
then the Pursuer has a move to a position $w_P$ so that $(w_P, q_E)$ belongs to the sequence.  By definition of a removable vertex ordering, the position each subsequent move by the Pursuer is closer
to the beginning of the sequence.  Hence, eventually a position in $\mathcal{F}$ is encountered.
\end{proof}

\end{document}